\documentclass[11pt,a4paper,reqno]{amsart}
\usepackage{epsfig}
\usepackage{amsfonts}
\usepackage{amsmath}
\usepackage{amssymb}
\usepackage{amsthm}
\usepackage{times}
\usepackage{color}
\usepackage{enumerate}
\usepackage[T1]{fontenc}
\usepackage[latin1]{inputenc}
\usepackage{graphicx}
\usepackage{nicefrac}
\usepackage[margin=25mm]{geometry}
\usepackage[all]{xy}
\usepackage{tikz-cd}

\usepackage{pgfplots}
\pgfplotsset{compat=1.15}
\usepackage{mathrsfs}
\usetikzlibrary{arrows}

\newtheorem{theorem}{Theorem}[section]
\newtheorem*{theorem*}{Theorem}

\newtheorem{lemma}[theorem]{Lemma}
\newtheorem{proposition}[theorem]{Proposition}

{\theoremstyle{remark}
 \newtheorem{remark}[theorem]{Remark}}
{\theoremstyle{definition}
 \newtheorem{definition}[theorem]{Definition}

 \newtheorem{example}[theorem]{Example}

}

\def\S{\mathcal{S}}

\newcommand{\CC}[0]{\ensuremath{\mathbb{C}}}
\newcommand{\ZZ}[0]{\ensuremath{\mathbb{Z}}}

\newcommand{\NN}[0]{\ensuremath{\mathbb{N}}}
\newcommand{\RR}[0]{\ensuremath{\mathbb{R}}}

\newcommand{\ra}{\rangle}
\newcommand{\la}{\langle}

\newcommand{\LND}[0]{\ensuremath{\operatorname{LND}}}

\newcommand{\spec}[0]{\ensuremath{\operatorname{Spec}}}

\newcommand{\cone}[0]{\ensuremath{\operatorname{cone}}}

\newcommand{\Der}[0]{\ensuremath{\operatorname{Der}}}

\newcommand{\li}{\varprojlim}

\newcommand{\bs}{\backslash}

\makeatletter \newcommand*\bigcdot{\mathpalette\bigcdot@{.5}} \newcommand*\bigcdot@[2]{\mathbin{\vcenter{\hbox{\scalebox{#2}{$\m@th#1\bullet$}}}}} \makeatother

\begin{document}

\title{Derivations on algebras of one-point compactification of affine semigroups}

\author{Roberto D\'iaz}
\address{Instituto de Matem\'atica y F\'\i sica, Universidad de Talca,
 Casilla 721, Talca, Chile.}
\email{robediaz@utalca.cl}


\thanks{
 \mbox{\hspace{11pt}}{\it Key words and phrases}: affine semigroup, locally nilpotent derivation, topologically integrable derivation, topological semigroup.\\
 \mbox{\hspace{11pt}}The author was partially supported by CONICYT-PFCHA/Doctorado Nacional/2016-folio 21161165.}

\begin{abstract}
For any affine semigroup $S$ the set $S\cup\{\infty\}$ has a natural structure of semigroup, additionally if $S$ is endowed with the discrete topology, the semigroup $S\cup\{\infty\}$ can be studied as the one-point compactification of $S$. In this article we study the derivations on the semigroup algebra $\CC[S\cup\{\infty\}]$ in relation to the derivations on the semigroup algebra $\CC[S]$ considering the metrizable topology on $\CC[S\cup\{\infty\}]$ induced by the one-point compactification topology of $S\cup\{\infty\}$.
\end{abstract}

\maketitle

\section*{Introduction}
The affine semigroups are related with the algebraic geometry since the duality with the category of affine toric variety \cite{CLS11}. This duality allow us to study geometry aspect of affine toric variety in term of semigroup properties, so toric varieties are an important testing ground for developing theories in algebraic geometry. An example of this interaction that motivates this research, is described as follow:

If $B$ is a $\CC$-algebra, a derivations on $B$ is an endomorphisms $\partial\colon B\rightarrow B$ such that $\partial(fg)=f\partial(g)+g\partial(f)$ for all $f$ and $g$ in $B$. A derivation $\partial$ on $B$ is said to be locally nilpotent derivation if for all $f\in B$ there exists $n:=n(f)\in \NN$ such that $\partial^{(n)}(f)=0$. This special type of derivation on $B$ is in correspondence with regular $(\CC,+)$-action on $\spec(B)$, an extensive study of the correspondence is developed in \cite{F06}. When $B=\CC[S]:=\bigoplus_{a\in M}C\cdot\chi^a$ is the $M$-graduated $\CC$-algebra associated to the affine semigroup $S$ with $M=\ZZ S$ also $C=\mathbb{C}$ when $a\in S$ and $C=\{0\}$ when $a\notin S$ (see Section \ref{graduated}), a homogeneous derivations on $\CC[S]$ is a derivation $\partial \colon \CC[S]\rightarrow \CC[S]$ such that there exists $e\in M$ verifying $\partial(C\cdot\chi^a)$ is sent to $C\cdot\chi^{a+e}$. The element $e\in M$ is the degree of $\partial$ and it is denoted by $e=\deg\partial$. In \cite[Lemma 1.10]{L10}
the author provides a proof to describe all derivation $\partial$ on $\CC[S]$ as a finite sum $\sum_{e\in M}\partial_{e}$ of homogeneous derivation $\partial_e$ and a complete description of homogeneous locally nilpotent derivation on $\CC[S]$ when $S$ is in correspondence with normal affine toric variety \cite[Theorem 2.7]{L10}. Geometrically, homogeneous locally nilpotent derivation on $\CC[S]$ are in correspondence with regular $(\CC,+)$-action on $\spec(\CC[S])$ normalized by the torus action.

Recently locally nilpotent derivation was generalized to the context of affine ind-scheme, see \cite{Kam96, Kam03} for an ample study of affine ind-scheme and \cite{DDL21} for the generalization of locally nilpotent derivation. In our context, if $B$ is a separated topological $\CC$-algebra with a fundamental system $\{\mathfrak{a}_i\}_{i\in \NN}$ of neighborhood of 0 consisting of ideals $\mathfrak{a}_i$, a $\CC$-continuous derivation on $B$ is said to be topologically integrable if the sequence of $\mathbf{\CC}$-linear endomorphisms $(\partial^{(n)})_{i\in \mathbb{N}}$ of $B$ converges continuously to the zero homomorphism, that is, if for every $f\in B$ and every $i\in\mathbb{N}$, there exist indices $n_0,j\in \mathbb{N}$ such that
$\partial^{(n)}(f+\mathfrak{a}_j)\subset \mathfrak{a}_i$ for every integer $n\geq n_0$, see Definition \ref{topologically integrable} or Lemma \ref{equivalence} for an equivalent definition in our context.

With an affine semigroup $S$ we can define naturally a new semigroup $S_\infty:=S\cup\{\infty\}$ and the respective semigroup algebra $\CC[S_\infty]$, if additionally $S$ is endowed with the discrete topology we can see $S_\infty$ as the one-point compactification of $S$ but, unlike $S$, the semigroup $S_\infty$ in general it is not a topological semigroup, that is, the associative binary operation $S_\infty\times S_\infty\rightarrow S_\infty$ in general it is not continuous. Theorem \ref{theorem topological semigroup}, shows that $S_\infty$ is a topological semigroup if and only if $\{0\}$ is a face of $S$, affine semigroups $S$ where $\{0\}$ is a face will be called pointed affine semigroups.

In this article we are interesting in to study the derivations on $\CC[S_\infty]$ in relation with the derivations on $\CC[S]$. With out topology, in Theorem \ref{correspondence} we prove the correspondence between derivation of $\CC[S]$ with $\CC[S_\infty]$ and as principal result, considering the topology on $\CC[S_\infty]$ induced by the one-point compactification topology of $S_\infty$, we make a complete classification of topologically integrable derivation on $\CC[S_\infty]$ in correspondence with homogeneous derivation on $\CC[S]$ developed in Section \ref{Section:4} and concluded in Theorem \ref{clasification}.\\

The article is organized as follow. In Section \ref{Section:1} we make a summary of necessary definitions and results used in this article. Section \ref{Section:2} is devoted to study the topology structure of one-point compactificated affine semigroup and the respective $\CC$-algebra. In Section \ref{Section:3} we prove the correspondence between derivation on $\CC[S]$ and $\CC[S\cup \{\infty\}]$ and finally, in Section \ref{Section:4}, applying the correspondence, we make a complete classification of topologically integrable derivation on $\CC[S_\infty]$ in relation to homogeneous derivation on $\CC[S]$.

\section{Preliminaries}\label{Section:1}

\subsection{Semigroup and topological semigroup} By a \emph{semigroup} we mean to a nonempty set $S$ with an associative binary operation $+\colon S\times S\rightarrow S$ where $+(a,b)$ will be denoted by $a+b$. An element $0$ in a semigroup is an \emph{identity} if for every element $a$ in $S$, the equations $0+a=a$ and $a+0=a$ hold. A semigroup is said to be \emph{commutative} if $a+b=b+a$ for all $a,b \in S$. Typical example of semigroup is the natural number $\mathbb{N}$, for us, the element $0$ will be in $\mathbb{N}$. A \emph{face} $F$ of a semigroup $S$ is a subsemigroup such that if $a+b$ is in $F$ then $a$ and $b$ are elements of $F$. Morphisms in the category of semigroup are functions preserving the semigroup structure, this morphism will be called semigroup homomorphisms. 
\emph{cancellative} semigroup are semigroup such that if $a+b=a+c$ then $b=c$. In what follows all semigroups are cancellative and commutative since commutative semigroups $S$ are embeddable in a group $M$ if and only if $S$ is cancellative (see \cite[theorem 3.10]{A13}). 

A \emph{Semitopological semigroup} is a semigroup $S$ endowed with a topology such that for all $a\in S$ the function $\lambda_a\colon S\rightarrow S$ defined by $b\mapsto a+b$ is continuous. More strong, a \emph{topological semigroup} is a semigroup $S$ endowed with a topology such that associative binary operation $+\colon S\times S\rightarrow S$ is continuous, where $S\times S$ is endowed with the product topology. All topological semigroup are semitopological semigroup but the converse in general is false has we can see in the example proposed in \cite[Exercise 2.1.2]{H11} or Example \ref{not topological} in this article. Morphisms in the category of semitopological semigroup and topological semigroup are continuous semigroup homomorphisms.

\subsubsection{Affine semigroup and convex geometry}

\emph{Affine semigroups} are commutative semigroups $S$ with an identity and a finite set $A\subset S$ such that $S=\{\sum_{a\in A}\alpha_aa\mid \alpha_a\in \mathbb{N}\}$ verifying that $S$ is embeddable in a lattice $M\simeq \ZZ^{n}$ for some $n\geq 0$ integer, particularly we will assume $M\simeq \ZZ S=\{a-b\mid a,b \in S\}$ where $a-b$ represent an equivalence class where the equivalence relation $\sim$ is defined on $S\times S$ as $(a,b)\sim (c,d)$ if and only if $a+d=c+b$. We identify $S$ in $M$ via $a\mapsto a-0$. Together to $M$ we associate to $S$ the lattice $\ZZ^n\simeq N=\operatorname{Hom}(M,\ZZ)$ and a bilinear function $\la\ ,\ \ra\colon M\times N\rightarrow \ZZ$ naturally expandable to bilinear function $\la\ ,\ \ra\colon M_\RR\times N_\RR\rightarrow \ZZ$ where $M_\RR$ and $N_\RR$ are the vectorial spaces $M\otimes \RR\simeq \RR^n$ and $N\otimes \RR\simeq \RR^n$ respectively. For any affine semigroup $S$ with finite generator set $A=\{a_1,\dots, a_n\}$ we define the convex polyhedral cone $\cone(A)=\{\lambda_1a_1+\dots +\lambda_na_n\in M_\RR\mid \lambda_i\geq 0\}$ and $\cone(A)^\vee=\{u\in N_\RR\mid \la m,u\ra\geq 0\ \text{for all}\ m\in \cone(A)\}$ this also is a convex polyhedral cone verifying $(\cone(A)^\vee)^\vee=\cone(A)$ (\cite[Proposition 1.2.4]{CLS11}). A face of a convex polyhedral cone $\sigma$ is a subset $\tau_m=\{u\in \sigma\mid \la m ,u\ra=0\}$ where $m\in \sigma^\vee$, convex polyhedral cones where $\{0\}$ is a face are called \emph{strongly convex} or pointed.
The dimension of a face $\tau$ is the dimension of the vectorial space generated by $\tau$. The set of faces of dimension $1$ of a convex polyhedral cone $\sigma$ will be denote as $\sigma(1)$ and the shortest vector in $N$ (or $M$) generator of a face $\rho \in \sigma(1)$ will be denoted indistinctly by $\rho$. Faces of polyhedral cones $\sigma$ are in correspondence with faces of the dual polyhedral cone $\sigma^\vee$ via $\tau\mapsto \tau^*:=\tau^{\perp}\cap \sigma^\vee$ where $\tau$ is a face of $\sigma$ and $\tau^{\perp}:=\{m\in M_\mathbb{R}\mid \la m,u\ra=0\text{ for all } u\in \tau.\}$ 
An affine semigroup is said to be $\emph{saturated}$ if $k a\in S$ when $k$ is a positive integer and $a\in M$, then $a\in S$. When $S$ is not saturated, we can associate to $S$ the affine semigroup $S^{sat}=\{a\in M\mid k a\in S,\text{ where $k$ is a positive integer}\}$ called the saturation of $S$. Additionally $S^{sat}=M\cap \operatorname{cone}(A)$ where $A$ is a finite generator set of $S$. 

\subsection{$\CC$-algebra of semigroup}\label{graduated}

 Let $S$ be a semigroup, a $S$-graded $\CC$-algebra is an associative, commutative $\CC$-algebra with unity and a direct sum decomposition $B=\bigoplus_{a\in S}B_a$ where $B_a\subset B$ are $\CC$-subvectorial spaces and $B_{a}B_{b}\subset B_{a+b}$. The $B_a$ are $S$-homogeneous component and we will say that the elements $f$ in $B_a$ for some $a\in S$ are $S$-homogeneous element of degree $a$. So for all semigroup $S$ we can associate it a $S$-graded $\CC$-algebra defined by $\CC[S]=\bigoplus_{a\in S}\CC\cdot\chi^a$ where $\chi^{a}$ are formal elements, verifying $\chi^a\chi^b=\chi^{a+b}$ and the identity is the element $1=\chi^0$. Morphisms between $\CC$-algebra of semigroup are homomorphisms of $\CC$-algebras induced by semigroup homomorphisms i.e. $\phi\colon \CC[S]\rightarrow \CC[S']$ is a morphism if there exists a semigroup homomorphism $\psi\colon S\rightarrow S'$ such that $\phi(\chi^a)=\chi^{\psi(a)}$. If $M$ is a group such that $S$ is embeddable in $M$ the $\CC$-algebra $\CC[S]$ also is $M$-graded via the identification $\CC[S]=\bigoplus_{a\in M}B_a$ where $B_a=\CC\cdot\chi^a$ if $a\in S$ and $B_a=\{0\}$ for the other case.

\subsection{Derivations}\label{derivation}
Let $B$ be a $\CC$-algebra, a $\CC$-derivation on $B$ is a $\CC$-linear function $\partial\colon B\rightarrow B$ such that $\partial(f)=0$ for all $f\in \CC$ and $\partial(fg)=f\partial(g)+g\partial(f)$ for all $f,g\in B$ (\emph{Leibniz's rule}). A derivation on $B$ is said to be locally nilpotent derivation (lnd for short) if for all $f\in B$ there exists $n\in \NN$ depending of $f$ such that $\partial^{(n)}(f)=0$ where $\partial^{(n)}$ is the $n$-th composition of $\partial$. The set of derivations on $B$ will be denoted by $\Der(B)$ and by $\LND(B)$ we mean to the set of lnd in $\Der(B)$ additionally the set of $\CC$-derivations on $B$ will be denoted by $\Der_\CC(B)$ and by $\LND_\CC(B)$ we mean to the set of lnd in $\Der_\CC(B)$.

If $B=\bigoplus_{a\in S}B_a$ is a $S$-graded $\CC$-algebra for some semigroup $S$, a derivation $\partial\in \operatorname{Der}(B)$ is said to be \emph{homogeneous} if it sends $S$-homogeneous elements into $S$-homogeneous elements, particularly when $S$ is affine semigroup, $M\simeq \ZZ^n$ the lattice associated to $S$ and $\CC[S]=\bigoplus_{a\in M}B_a$ is the $M$-graded previously described, for every nontrivial homogeneous derivation $\partial$ there exists an element $e\in M$ such that $\partial(B_a)\subset B_{a+e}$ for all $a\in M$ as is proved in the follow lemma:

\begin{lemma}
Let $M$ be the lattice associated to an affine semigroup $S$. If $\partial\colon \bigoplus_{a\in M}B_a\rightarrow \bigoplus_{a\in M}B_a$ is a homogeneous derivation, then there exists $e\in M$ such that $\partial(B_a)\subset B_{a+e}$ for all $a\in M$.
\end{lemma}
\begin{proof}
Let $a\neq b$ be an element in $M$ and let $f$, $g$ be elements in $B_{a}$ and $B_{b}$ respectively such that $\partial(f)\neq0\in B_{a+e(a)}$ and $\partial(g)\neq 0\in B_{b+e(b)}$, where $e(a)$ and $e(b)$ are elements in $M$ depending of $a$ and $b$ respectively. By Leibniz's rule $\partial(fg)=f\partial(g)+g\partial(f)\in B_{b+a+e(a)}\bigoplus B_{a+b+e(b)}$ but also $\partial(fg)\in B_{a+b+e(a+b)}$ where $e(a+b)$ is an element in $M$ depending of $a+b$, we conclude that 
$a+b+e(b)=b+a+e(a)$ thus
$e(a)=e(b)$.
\end{proof}
The element $e\in M$ associated to $\partial$ will be called the \emph{degree of} $\partial$ and will be denoted by $e=\deg \partial$. We will say homogeneous derivation instead of $S$-homogeneous derivation when the context is clear.

 \subsection{Topological structures}
 A \emph{topological ring} $R$ is a ring with topological structure such that the functions $R\times R\rightarrow R$ defined by $(f,g)\mapsto f-g$ and $(f,g)\mapsto fg$ are continuous, morphisms between topological rings are continuous homomorphisms between the rings. A \emph{Topological $R$-algebra} on a topological ring $R$ is a topological ring $B$ such that the function induced since the structure of $R$-algebra $R\times B \rightarrow B$ is continuous, also the morphism between topological $R$-algebras are continuous homomorphism between $R$-algebras.
 A topology on a topological ring $R$ (or $R$-algebra $B$) will be called linear if there exists a fundamental system of open neighborhood of its additive identity element $0$. In this article we will assume $R=\CC$ with discrete topology and $B$ a separated topological $\CC$-algebra with fundamental system of open neighborhood of its neutral element $0$ consisting in a countable family of ideals $\mathfrak{a}_0\supset\mathfrak{a}_1\supset \dots$.

\subsubsection{Topologically integrable derivation.}
Topologically integrable derivation was introduced by \cite{DDL21}
as a generalization of locally nilpotent derivation in the context of topological $R$-algebras and affine ind-schemes, theory develop for example in \cite{Kam96,Kam03}. This generalization is motivated since geometric aspect because, in the same form that locally nilpotent derivation, the topologically integrable derivations are in correspondence with $\mathbb{G}_{R,a}$-actions on the formal spectrum of a topological $R$-algebra $B$ (see \cite[Section 3]{DDL21}). In this article our interest is on topological $\CC$-algebras, there exists a more general definition, but for this article we work with the simplified definition:

\begin{definition}\label{topologically integrable}
Let $B$ be a topological $\CC$-algebra. We call a continuous $\CC$-derivation $\partial$ of $B$ \emph{topologically integrable} if the sequence of $\mathbf{k}$-linear endomorphisms
$(\partial^{(n)})_{i\in \mathbb{N}}$ of $B$ converges continuously to the zero homomorphism, that is, if for every $f\in B$ and every
$i\in\mathbb{N}$, there exists an indices $n_0,j\in \mathbb{N}$ such that
$\partial^{(n)}(f+\mathfrak{a}_j)\subset \mathfrak{a}_i$ for every integer $n\geq n_0$.
\end{definition}
When $\mathfrak{a}_i=\{0\}$ for all $i\in \NN$ the topology on $B$ is the discrete topology and the definition of topologically integrable derivation coincide with locally nilpotent derivation. Allow us a natural example to clarify the topologically integrable derivation.

\begin{example}
Let $\CC[[u]]$ the ring of formal power series endowed with the $u$-adic topology, the $u$-adic topology is induced by the fundamental system of open ideal $\CC[[u]]=\la 1\ra\supset \la u\ra\supset \la u^2\ra\supset\dots$ where $\la u^j\ra$ are the ideals generated by $u^j$ for all integer $j\geq 0$. We define the derivation $\partial=u^2\dfrac{d}{du}\colon \CC[[u]]\rightarrow \CC[[u]]$. The derivation $\partial$ is continuous since $\partial(\la u^j\ra)\subset \la u^{j+1}\ra$, also is topologically integrable since $\partial^{(n)}(f+\la u^j\ra)\subset \la u^{n+1} \ra$ thus for every $f\in \CC[[u]]$ and $i\in n$, with arbitrary $j$ and $n_0=i$ it is verified that $\partial^{(n)}(f+\la u^j\ra)\subset \la u^i \ra$ for all $n\geq n_0=i$. A similar argument verifies that $u^l\dfrac{d}{du}\colon \CC[[u]]\rightarrow \CC[[u]]$ is topologically integrable for all $l\geq 2$.
\end{example}

As a rewriting of \cite[Lemma 1.9]{DDL21} we can conclude that:
\begin{lemma}\label{equivalence}
A derivation $\partial\in \Der_{\CC}(B)$ is topologically integrable if and only if the follows are verified:
\begin{enumerate}
 \item[(P.1)] The sequence $(\partial^{(n)}{(f)})_{n\in \NN}$ converge to $0$ for all $f\in B$ and
 \item[(P.2)] For all $i\in \NN$ there exists $j\in \NN$ such that $\partial^{(n)}(\mathfrak{a}_j)\subset \mathfrak{a}_i$ for all $n\in \NN$.
\end{enumerate}
\end{lemma}

\section{one-point compactification of affine semigroup}\label{Section:2}

Let $S$ be a semigroup we can define the semigroup $S_\infty$ associated to $S$ as the set $S\cup \{\infty\}$ with the associative binary operation induced by the binary operation of $S$ and $a+\infty=\infty+a=\infty$ for all $a\in S_{\infty}$. The construction give us a commutative diagram:

\[\begin{tikzcd}
S\times S \arrow{r}{\eta_+} \arrow[swap]{d}{} & S \arrow{d}{} \\
S_\infty\times S_\infty \arrow{r}{\zeta_+} & S_\infty
\end{tikzcd}
\]

where $\zeta_+|_{S\times S}=\eta_+$. The notation $\zeta_+$ and $\eta_+$ will be used to describe the binary operations instead of
$+$ when we need to differentiate between the binary operations on $S$ and $S_\infty$.

If $S$ is a Hausdorff and locally compact topological semigroup we will endow $S_\infty$ with the one-point compactification topology (Alexandroff extension), i.e. the subsets $U$ in $S_\infty$ are open if:
\begin{enumerate}
 \item $U\subset S$ and is an open in $S$,
 \item $U=S_{\infty}\bs C$ where $C$ is compact in $S$.
\end{enumerate}
Equivalently the subsets $F$ in $S_\infty$ are closed if:
\begin{enumerate}
 \item $F=F'\cup\{\infty\}$ where $F'$ is a closed set in $S$,
 \item $F$ is a compact set in $S$.
\end{enumerate}
additionally $S_\infty$ verify to be Hausdorff, compact. In this section always we will assume $S$ is an affine semigroup endow with the discrete topology and $S_\infty$ will be endowed with the one-point compactification topology, easily we can see that $S_\infty$ is a semitopological semigroup as is proved in Lemma \ref{semitopological} but in general it is not a topological semigroup as we can see in Example \ref{not topological}.

Although some proofs prove more general cases, but for our purpose, the above conditions are sufficient.

\begin{lemma}\label{semitopological}
The semigroup $S_\infty$ is a semitopological semigroup.
\end{lemma}
\begin{proof}
We need to study the options $a=0$, $a=\infty$ and other cases. 
\begin{enumerate}
 \item $a=0$. The function $\lambda_0\colon S_\infty\rightarrow S_\infty$ is the identity function so preimage of closed sets are closed sets.
 \item $a=\infty$. for all subset $A\subset S_\infty$ the preimage by $\lambda_\infty$ is $S_\infty$ if $\infty\in A$ or the empty set if $\infty\notin A$ thus also is continuous.
 \item Other cases. If $F$ is a compact in $S\subset S_{\infty}$ then is finite, additionally $S$ is cancellative so preimage of $F$ by $\lambda_a$ is finite thus compact in $S$ and closed in $S_\infty$. The other possibility is $F=F'\cup\{\infty\}$ for some closed $F'\subset S$, the preimage of $F$ contain $\infty$ and since $S$ is endowed with the discrete topology any sub set is closed thus the preimage of $F$ is closed in $S_\infty$.
\end{enumerate}

\end{proof}

Generally not true that for some affine semigroup $S$ endowed with the discrete topology the semigroup $S_\infty$ with the one-point compactification topology is a topological semigroup.

\begin{example}\label{not topological}
Let $S=\ZZ$ be the semigroup (group) of integer number. The binary function $+\colon S_\infty \times S_\infty\rightarrow S_\infty$ is not continuous in $(\infty,\infty)$ because for example if we fix the open $S_\infty\bs\{0\}$, neighborhood of $\infty+\infty=\infty$, any basis open $U=(U_1\cup\{\infty\})\times(U_1\cup\{\infty\})$ 
has an element $(a,-a)$ for some $a\in \ZZ$ then $+(U)\subsetneq S_\infty\bs\{0\}$.
\end{example}

Conversely, if $S=\NN$ the semigroup $S_{\infty}$ as we will prove more generally in Theorem \ref{theorem topological semigroup}, it is a topological semigroup.

Recall a \emph{Face} of a semigroup $S$ is a subsemigroup $F$ such that if $a+b\in F$ then $a\in F$ and $b\in F$.

\begin{definition}
By a pointed affine semigroup $S$ we mean to an affine semigroup $S$ where $\{0\}$ is a face of $S$, equivalently if $a+b=0$ then $a=b=0$.
\end{definition}

\begin{lemma}\label{finite1}
If $S\simeq \NN^l$ for some $l$ then for all $a\in S$ there exist finite pairs $(b,c)$ in $S\times S$ such that $a=b+c$.
\end{lemma}
\begin{proof}
We can see if $a=(a_1,\dots,a_l)=\alpha_1e_1+\dots +\alpha_l e_l$ where $e_i$ is the lattice vector with $0$ in its components except for the $i$-th component which has a $1$ and $\alpha_i$ in $\mathbb{N}$ for all integer $0\leq i\leq l$, then $\alpha_i\leq a_i$ so there exist finite combinations for obtain $a$ thus there exist finite pairs $(b,c)\in S\times S$ such that $a=b+c$.
\end{proof}

\begin{proposition}\label{finite}
$S$ is a pointed affine semigroup if and only if for all $a$ in $S$ there exist finite pairs $(b,c)$ such that $a=b+c$.
\end{proposition}
\begin{proof} 
First we assume $S$ a pointed affine semigroup. If $S$ is a pointed affine semigroup the respective convex polyhedral cones $\sigma^\vee\subset M_\RR$ is pointed and full dimensional thus $\sigma\subset N_\RR$ also is pointed and full dimensional (\cite[Proposition 1.2.12]{CLS11}). If $\sigma(1)=\{r_1,\dots,r_l\}$, we define the homomophism $P\colon\NN^l\rightarrow N$ defined by $e_i\mapsto r_i$ and a dual injective homomorphism $P^*\colon M\rightarrow \ZZ^l$ such that $P^*(S)\subset \NN^l$ (see \cite[Section 2.1]{ADHL15}).
 We can see $b+c=a$ in $S$ if and only if $P^*(b)+P^*(c)=P^*(a)$ in $\NN^l$, by Lemma \ref{finite1} we conclude that for all $a$ in $S$ there exist finite pairs $(b,c)$ such that $a=b+c$. If $S$ is not pointed there exist $a_1\neq 0$ and $a_2\neq 0$ such that $a_1+a_2=0$ thus for all $b,c\in S$ it is verified that $b+c=(b+l a_1)+(l a_2+c)$ for all $l$ in $\NN$.

\end{proof}

\begin{theorem}\label{theorem topological semigroup}
Let $S$ be an affine semigroup endowed with the discrete topology. The semigroup $S_{\infty}$ is a topological semigroup if and only if $S$ is pointed.
\end{theorem}
\begin{proof}
First we will assume $S$ a pointed affine semigroup. Let $F$ be a closed set in $S_\infty$ of type $F\subset S$ and compact in $S$ so $F=\{a_1,\dots a_l\}$ and $\zeta_+^{-1}(F)=\eta_+^{-1}(F)=\cap_{i=1}^l \eta_+^{-1}(\{a_i\})$, by Proposition \ref{finite} $\eta_+^{-1}(\{a_i\})$ is finite or $\eta_+^{-1}(\{a_i\})=\cup_{j=1}^m\{b^{(i)}_j\}\times\{c^{(i)}_j\}$ where $b^{(i)}_j+c^{(i)}_j=a_i$ for all $1\leq j\leq m$ integer, thus $\zeta_+^{-1}(F)$ is a closed in $S_\infty\times S_\infty$. Now we assume $F$ a closed of type $F'\cup\{\infty\}$ where $F'$ is a closed in $S$ and let $U=S\bs F'$. Since $\eta_+$ is surjective because for all $a\in S_\infty$ we have $\eta_+(0,a)=a$, just it is necessary to verify that $\eta_+^{-1}(U)$ is an open in $S_\infty\times S_\infty$. Since definition of the topology of $S_\infty$, the sets $\{a\}\subset S\subset S_\infty$ are open so $\{a\}\times \{b\}$ is an open in $S_\infty\times S_\infty$ for all $a,b\in S$ also since the definition of $\eta_+$, the preimage of $U$ is a subset of $S\times S$ thus is the union of open of the form $\{a\}\times \{b\}$ and we conclude that $U$ is an open and $F$ is a closed. If $S$ is not pointed, there exists an isomorphic copy of $\ZZ$ in $S$ so with a similar argument of Example \ref{not topological} we conclude the theorem.
\end{proof}
In what follows we will assume $S$ pointed affine semigroup so when $S$ is endowed with the discrete topology, $S_\infty$ also will be a topological semigroup endowed with the one-point compactification topology. Now we are interested in to study a topology on $\mathbb{C}[S]$ and $\mathbb{C}[S_\infty]$ preserving some properties of $S$ and $S_\infty$ respectively. If we study topological characteristics of $S_\infty$ we can see that $S_\infty$ is compact, Hausdorff and totally disconnected thus $S_\infty$ is the projective limit of finite semigroups endowed with the discrete topology \cite[Theorem 1]{N57}. An element $a\neq 0$ in $S$ is said to be \emph{irreducible} if $a=b+c$ then $b=a$ or $c=a$, the set of irreducible element of $S$ commonly is denoted by $\mathscr{H}$ and called the Hilbert basis of $S$, additionally $\mathscr{H}$ is a finite set as is proved for example in \cite[Proposition 1.2.23]{CLS11}. We want to describe the finite semigroups in term of Hilbert basis of $S$ that the projective limit give us $S_\infty$.

Let $S$ be an affine semigroup with $\mathscr{H}=\{a_1,\dots,a_s\}\subset S$ Hilbert basis of $S$, by simplicity, we define the function $\mathfrak{s}:S\rightarrow \NN$ defined as $\mathfrak{s}(a)=\operatorname{max}\{\sum_{k=1}^s \alpha_k \mid a=\sum_{k=1}^s \alpha_ka_k$\}. Now we define the sets for all $i\in \NN$, $H_i=\{a \in S\mid \mathfrak{s}(a)\leq i\}$, for example $H_0=\{0\}$ and $H_1=\mathscr{H}\cup\{0\}$. With $H_i$ for all $i\in \NN$ we define the finite semigroup $S_i=H_i\cup\{\infty\}$ where the structure of semigroup is given by: for $a,b\in H_i$, $a+b=\infty$ if $a+b\notin H_i$ and $a+\infty=\infty+a=\infty$ for all $a\in S_i$. The surjective homomorphism $\phi_i\colon S_{i+1}\rightarrow S_i$ are defined by $\phi_i(a)=a$ if $a\in S_i$ and $\phi_i(a)=\infty$ for the other case.

\begin{example}\label{natural compactificated}
Let $S_i=(\{0,1,2,\dots,i,\infty\},+)$ be a topological semigroup endowed with the discrete topology and the semigroup structure is defined by $a+b=\infty$ if $a+b>i$ and $a+\infty=\infty+a=\infty$.

 The semigroup homomorphisms $\varphi_i\colon S_{i+1}\rightarrow S_i$, defined naturally, give us the projective limit $\S=\li S_i$.
$\S$ is isomorphic, as topological semigroup to $\NN\cup\{\infty\}$ where $\NN\cup\{\infty\}$ is the one-point compactification of $\NN$ with the discrete topology.
\end{example}

\begin{lemma}
 Let $S$ be a pointed affine semigroup and $S_i$ described previously then $S_\infty\simeq \li S_i$ as topological semigroup.
\end{lemma}
\begin{proof}
A general element in $\li S_i$ can be expressed it of the form $(a_0,a_1,\dots)$. We have the next options:
\begin{enumerate}
 \item $0=a_0=a_1=a_2=\cdots$,
 \item $\infty=a_0=a_1=a_2=\cdots$,
 \item $\infty=a_0=a_1=a_2=\cdots=a_{l-1}$ and $a_{l}=a_{l+1}=\cdots$ for some $a_{l}\in S$.
\end{enumerate}

So every element $a=(a_0,a_1,\dots)\in \li S_i$ has a component $a_l \in S_\infty$ such that $a_{l}=a_{l+1}=\cdots$. The semigroup homomorphism $\li S_i\rightarrow S_\infty$ is defined by $a=(a_0,a_1,\dots)\mapsto a_l$ where $a_l$ is a component of $a$ such that $a_l=a_{l+1}=\dots$, this component will be called the distinguished component, obviously the function defined is an isomorphism of semigroup. 
 Respect to the topology, for all $a_0\in S$ the set $(\prod_{i\in \NN} U_i)\cap \li S_i$ where $U_i=\{a_0\}$ for all $0\leq i \leq l$ integer and $U_i=S_i$ for all $i>l$ is an open and equal to $\{(a_0,a_0,\dots)\}$ so, every set $U\subset \li S_i$ with out $(\infty,\infty,\dots)$ is an open. On the other hand, if $U=(\prod_{i\in \NN} U_i)\cap \li S_i\neq \li S_i$ is an open with the element $(\infty,\infty,\dots)$, there exists a $l$ such $U_l\neq S_l$ and $U_i=S_i$ for all $i>l$. If $A_i=S_i\bs U_i$ with $i\leq l$ we conclude that $\li S_i\bs U$ is the set of element with component distinguished in $A_i$ for some $i\leq l$ thus $\li S_i\bs U$ is finite also compact and its topology coincides with the topology of $S_\infty$.
\end{proof}

The previous characterization of $S_\infty$ is motivated in to study $\mathbb{C}[S_\infty]$ with a topological structure. In particular, the projective system induced by $S_0\leftarrow S_1 \leftarrow S_2\leftarrow \dots$ where $\phi_i\colon S_{i+1}\rightarrow S_i$ was defined previously, gives us a projective system $\mathbb{C}[S_0]\leftarrow\mathbb{C}[S_1] \leftarrow \mathbb{C}[S_2]\leftarrow \dots$ of topological $\mathbb{C}$-algebras of semigroup where $\mathbb{C}[S_i]$ will be endowed with the discrete topology. To note also, we can define a the $\mathbb{C}$-homomorphisms of semigroup $\mathbb{C}$-algebras $\psi_i\colon \mathbb{C}[S_\infty]\rightarrow \mathbb{C}[S_i]$ and to obtain a liner system of ideals $\mathfrak{a}_0\supset\mathfrak{a}_1\supset \mathfrak{a}_2\supset\dots$ where $\mathfrak{a}_i=\ker\psi_i$ thus a topology on $\mathbb{C}[S_\infty]$ given by $\{\mathfrak{a}_i\}_{i\in \NN}$ and an identification of $\mathbb{C}[S_\infty]\simeq \mathbb{C}[\li S_i]$ in $\li \mathbb{C}[S_i]$ as sub topological $\mathbb{C}$-algebra via $\chi^{a}\mapsto (\chi^{\pi_0(a)},\chi^{\pi_1(a)},\dots)$ (see \cite[Chapter 9]{N68}), generally it is not true that there exists an isomorphic identification of $\mathbb{C}[\li S_i]$ with $\li \mathbb{C}[S_i]$ for arbitrary topological semigroup as is showed in \cite[Example 5.2]{DL21} or in the next example:
\begin{example}
We use the notation in Example \ref{natural compactificated}, and we will show an element in $\li\mathbb{C}[S_i]$ unidentifiable in $\mathbb{C}[\li S_i]\simeq \mathbb{C}[ S_\infty]$. The element is constructed as follow: $f_0=1+\chi^\infty$, $f_1=1+\dfrac{1}{2}\chi^1+\dfrac{1}{2}\chi^\infty$ the next element is $f_2=1+\dfrac{1}{2}\chi^1+\dfrac{1}{4}\chi^2+\dfrac{1}{4}\chi^\infty$ and as general element: $$f_l=\dfrac{1}{2^l}\chi^{\infty}+\sum_{k=0}^l\dfrac{1}{2^k}\chi^k$$
Thus the element $(f_0,f_1,\dots)$ is an element in $\li\mathbb{C}[S_i]$ with out an identification in $\mathbb{C}[ S_\infty]$ since in $\mathbb{C}[ S_\infty]$ just we have finite sums.
\end{example}

 We can see $\li \mathbb{C}[S_i]$ as the completion of $\mathbb{C}[S_\infty]$ respect to the linear system $\{\mathfrak{a}_i\}_{i\in \NN}$ and is usual in the literature $\li \mathbb{C}[S_i]$ will be denoted by $\widehat{\mathbb{C}[S_\infty]}$. In what follows, when $\mathbb{C}[S_\infty]$ has topological structure will be identified as sub topological $\mathbb{C}$-algebra of $\widehat{\mathbb{C}[S_\infty]}$ induced by the linear system $\{\mathfrak{a}_i\}_{i\in \NN}$.

\section{Locally nilpotent derivation on $\mathbb{C}[S_\infty]$}\label{Section:3}
In this section the algebraic object are used with out topology.
If we fix a general element $f=\sum_{k=1}^nc_k\chi^{a_k}+c\chi^\infty\in \mathbb{C}[S_\infty]$ with $c_1,\dots,c_n,c\in \mathbb{C}$ and $a_1,\dots, a_n\in S$ then $f\chi^\infty=0$ if and only if $\sum_{k=1}^n c_k+c=0$, particularly the ideal $I_\infty=\{f\in\mathbb{C}[S_\infty]\mid f\chi^\infty=0\}$ is a subring $I_\infty\subset\mathbb{C}[S_\infty]$ with identity $1-\chi^\infty$. The ideal $I_\infty$ will be fundamental to 
describe the derivations on $\CC[S_\infty]$ as is motivated in Example \ref{example derivation}, 
and to obtain the correspondence of $\Der(\mathbb{C}[S])$ with $\Der(\mathbb{C}[S_\infty])$. The existence of $\chi^\infty$ in $\mathbb{C}[S_\infty]$ together with the Leibniz's rule verified by the derivations, give us restriction on the possibility to define derivations. In the next lemma, we fix the more important restriction on the image of derivations.
 \begin{lemma}\label{0 divisor}
 If $\partial\in\Der(\mathbb{C}[S_\infty])$ then $\partial(\chi^a)=0$ or a $0$ divisor. Particularly $\partial(\chi^\infty)=0$ and $\partial(\chi^a)\in I_{\infty}$ when $a\neq \infty$.
 \end{lemma}
 
 \begin{proof}
 First we verify that $\partial(\chi^\infty)=0$. We can see that $\partial(\chi^\infty)=2\chi^\infty \partial(\chi^\infty)$ or $\partial(\chi^\infty)(2\chi^\infty-1)=0$ an easy computation proves that $(2\chi^\infty-1)$ is not a $0$ divisor so $\partial(\chi^\infty)=0$. Now $0=\partial(\chi^\infty)=\partial(\chi^\infty\chi^a)=\chi^\infty \partial(\chi^a)$ thus $\chi^\infty \partial(\chi^a)=0$ or equivalently $\partial(\chi^a)\in I_\infty$.
 \end{proof}
 
 The Lemma \ref{0 divisor} has directed consequences on the structure of the derivations, in the next proposition we make a summary of these consequences:

 \begin{proposition}\label{lemmas}
 If $\partial\in \Der(\mathbb{C}[S_\infty])$, then the following sentences are verified:
 \begin{enumerate}
 \item $\partial$ defines an element $\partial_\infty\in \Der(I_\infty)$ where $\partial_\infty$ is the natural restriction.
 \item $(1+c\chi^\infty)\partial(f)=\partial(f)$ for all $c\in \CC$ and for all $f\in \CC[S_\infty]$.
 \item If $\partial$ is a homogeneous derivation, then $\partial=0$.
 \end{enumerate}
 \end{proposition}

\begin{proof}
The proofs are immediate consequence of Lemma \ref{0 divisor}.
\end{proof}

\begin{example}\label{example derivation}
Let $S_\infty=\NN\cup\{\infty\}$ be a semigroup with $a+\infty=\infty$ for all $a\in S_\infty$.
The ring of semigroup verifies:
 $$\mathbb{C}[S_\infty]\simeq\mathbb{C}[x,y]/(xy-y,y^2-y)\simeq \mathbb{C}[\chi^1,\chi^{\infty}]$$
 
 In the last $\mathbb{C}$-algebra $\chi^1\chi^{\infty}=\chi^{\infty}$ and $\chi^{\infty}\chi^{\infty}=\chi^{\infty}$. We can define the derivation
 
\[
\begin{array}{rcl} 
 \partial\colon \mathbb{C}[S_\infty]& \rightarrow &\mathbb{C}[S_\infty] \\ 
 \chi^1&\mapsto &1-\chi^\infty\\
 \chi^\infty&\mapsto &0
 \end{array}
\]

It is easy to see $\partial$ is well defined as derivation such that $\partial\in \operatorname{LND}(\mathbb{C}[S_\infty])$ and since definition of $\partial$, it is not possible to obtain a decomposition in $S_\infty$-homogeneous derivations.

\end{example}

\begin{proposition}
Let $\partial,\partial'$ be derivations in $\Der(\mathbb{C}[S_\infty])$. Then $\partial=\partial'$ if and only if $\partial_\infty=\partial'_\infty$.
\end{proposition}
\begin{proof}
Let $\partial,\partial'$ be derivations in $\Der(\mathbb{C}[S_\infty])$ such that $\partial_\infty=\partial'_\infty$. If $a\in S$ then $\chi^a(1-\chi^\infty)\in I_\infty$, since $\partial_\infty=\partial'_\infty$, we have $\partial_\infty(\chi^a(1-\chi^\infty))=\partial'_\infty(\chi^a(1-\chi^\infty))$ or $\partial_\infty(\chi^a-\chi^\infty)=\partial'_\infty(\chi^a-\chi^\infty)$ thus $\partial_{\infty}(\chi^a)=\partial'_\infty(\chi^a)$ equivalently $\partial(\chi^a)=\partial'(\chi^a)$. The other direction is obvious.

\end{proof}

In general, we can see $\partial(\chi^a)=\sum_{k=1}^nc_k\chi^{a_k}+c\chi^\infty\in \mathbb{C}[S_\infty]$ with $c_1,\dots,c_n,c\in \mathbb{C}$ and $a_1,\dots, a_n\in S$. With this idea we are going to denote $f_{\partial(a)}=\sum_{k=1}^nc_k\chi^{a_k}$ and $c_{\partial(a)}=c$, if $a=\infty$ we fix $f_{\partial(\infty)}=0$ and $c_{\partial(\infty)}=0$ .

\begin{proposition}\label{translation}
With the previous notation. If $\partial\in \Der(\mathbb{C}[S_\infty])$, then the function $\partial'\colon \mathbb{C}[S]\rightarrow \mathbb{C}[S]$ defined by $\chi^a\mapsto f_{\partial(a)}$ verifies to be an element $\partial'\in \Der(\mathbb{C}[S])$.
\end{proposition}
\begin{proof}
To verify that $\partial'\in \Der(\mathbb{C}[S])$ we need to prove that $f_{\partial(a+b)}=f_{\partial(a)}\chi^{b}+f_{\partial(b)}\chi^{a}$ for all $a, b\in S$. $f_{\partial(a+b)}+c_{\partial(a+b)}\chi^\infty=\partial(\chi^{a+b})=\partial(\chi^{a}\chi^{b})=(f_{\partial(a)}+c_{\partial(a)}\chi^{\infty})\chi^{b}+(f_{\partial(b)}+c_{\partial(b)}\chi^{\infty})\chi^{a}=f_{\partial(a)}\chi^{b}+f_{\partial(b)}\chi^{a}+(c_{\partial(a)}+c_{\partial(b)})\chi^{\infty}$. Thus $\partial'\in \Der(\mathbb{C}[S])$.
\end{proof}

\begin{proposition}
If $\partial'\in \Der(\mathbb{C}[S])$ and $f\in I_\infty\subset \mathbb{C}[S_\infty]$, then the function $\partial=f\partial'\colon\mathbb{C}[S_\infty]\rightarrow\mathbb{C}[S_\infty]$ defined by $\partial(\chi^\infty)=0$ and $\partial(\chi^a)=f\partial'(\chi^a)$ is an element $\partial\in\Der(\mathbb{C}[S_\infty])$. Particularly any $\partial\in \Der(\mathbb{C}[S_\infty])$ verifies $\partial(\chi^{a})=(1-\chi^\infty)\partial'(\chi^a)$ for all $a\in S$ where $\partial'\in \Der(\mathbb{C}[S])$.
 \end{proposition}

\begin{proof}
To verify that effectively $\partial\in \Der(\mathbb{C}[S_\infty])$ just we need to verify the Leibniz's rule on $\chi^a\chi ^\infty=\chi^\infty$ for all $a\in S_\infty$. We assume $a\neq \infty$, when $a=\infty$ there is no problem. So $0=\partial(\chi^\infty)=\partial(\chi^a\chi^\infty)=\chi^\infty \partial(\chi^a)=\chi^\infty f \partial'(\chi^a)=0$. To verify the last part if $\partial\in\Der(\mathbb{C}[S_\infty])$ in the same form that Proposition \ref{translation} we define the element $\partial'\in \Der(\mathbb{C}[S])$. 
If $a\in S$, $(1-\chi^\infty)\partial'(\chi^{a})=(1-\chi^\infty)(\partial(\chi^a)-c_{\partial(a)}\chi^\infty)=\partial(\chi^a)-\chi^\infty \partial(\chi^a)-c_{\partial(a)}\chi^\infty+c_{\partial(a)}\chi^\infty=\partial(\chi^a)-\chi^\infty \partial(\chi^a)$ thus since Lemma \ref{0 divisor} $\partial(\chi^a)-\chi^\infty \partial(\chi^a)=\partial(\chi^a)$ and we conclude the last part of the proposition.

\end{proof}

\begin{theorem}\label{correspondence}
There exists an one to one correspondence between $\Der(\mathbb{C}[S])$ and $\Der(\mathbb{C}[S_\infty])$.
\end{theorem}
\begin{proof}
Since the previous propositions just we need to prove that any derivation $\partial'\in \Der(\mathbb{C}[S])$ verifies $\partial(\chi^a)=\partial'(\chi^a)+c_{\partial(a)}\chi^\infty$ for some derivation $\partial\in \Der(\mathbb{C}[S_\infty])$ and for all $a\in S$. Let $\partial'\in \Der(\mathbb{C}[S])$ be a derivation and let $\partial$ be the derivation defined by $\partial(\chi^a)=(1-\chi^\infty)\partial'(\chi^a)$ if $a\in S$ and $\partial(\chi^\infty)=0$. Thus $\partial(\chi^a)=(1-\chi^\infty)\partial'(\chi^a)=\partial'(\chi^a)-\partial'(\chi^a)\chi^\infty$, obviously $-\partial'(\chi^a)\chi^\infty=c_{\partial(a)}\chi^\infty$ since the definition of $\partial$.
\end{proof}

\begin{remark}
With the correspondence described we can see that $\Der(\mathbb{C}[S])$ and $\Der(\mathbb{C}[S_\infty])$ are isomorphics as $\CC$-vectorial spaces. Basically if $\partial_1$ and $\partial_2$ in $\Der(\mathbb{C}[S_\infty])$ we have that for all $a\in S$ it is verified $(\partial_1+\partial_2)(\chi^a)=\partial_1(\chi^a)+\partial_2(\chi^a)=(1-\chi^\infty)\partial'_1(\chi^a)+(1-\chi^\infty)\partial'_2(\chi^a)$, with $\partial'_1$ and $\partial'_2$ in $\Der(\mathbb{C}[S_\infty])$, then $(\partial_1+\partial_2)(\chi^a)=(1-\chi^\infty)(\partial'_1+\partial'_2)(\chi^a)$.

\end{remark}

\begin{proposition}\label{LND}
With the correspondence previously defined, $\partial\in\LND(\mathbb{C}[S_\infty])$ if and only if the associated derivation verifies $\partial'\in\LND(\mathbb{C}[S])$.
\end{proposition}
\begin{proof}

Let $\partial'\in \Der(\mathbb{C}[S])$ be a derivation. The associated derivation in $\Der(\mathbb{C}[S_\infty])$ defined by $\partial(\chi^\infty)=0$ and $\partial(\chi^a)=(1-\chi^\infty)\partial'(\chi^a)$ when $a\in S$ verifies $\partial^{(n)}(\chi^a)=(1-\chi^\infty) \partial'^{(n)}(\chi^a)$ thus if $\partial'$ is locally nilpotent derivation then $\partial$ is a locally nilpotent derivation. In the other direction if $\partial$ is locally nilpotent derivation, there exists $n$ such that $(1-\chi^\infty) \partial'^{(n)}(\chi^a)=0$ for all $a\in S$, thus $\partial'^{(n)}(\chi^a)-\partial'^{(n)}(\chi^a)\chi^\infty=0$ if and only if $\partial'^{(n)}(\chi^a)=0$ since $\partial'^{(n)}(\chi^a)$ is an element in $\CC[S]$.
\end{proof}

\section{Topologically integrable quasi-homogeneous derivation on $\mathbb{C}[S_\infty]$}\label{Section:4}
Since the correspondence proved in Theorem \ref{correspondence} and the impossibility to obtain homogeneous derivation on $\mathbb{C}[S_\infty]$ different of $0$ (Proposition \ref{lemmas}), we will say that a derivation on $\mathbb{C}[S_\infty]$ is \emph{ quasi-homogeneous} if the corresponding derivation on $\mathbb{C}[S]$ is homogeneous. The interest in quasi-homogeneous derivation on $\mathbb{C}[S_\infty]$ is analogous to the interest on homogeneous derivation of $\mathbb{C}[S]$, that is, the possibility to obtain a decomposition of any derivation $\partial$ as a finite sum of homogeneous derivations, a proof of this can be seen in \cite[Lemma 1.10]{L10}, so since the linear correspondence between derivations on $\mathbb{C}[S]$ and $\mathbb{C}[S_\infty]$ also any derivation on $\mathbb{C}[S_\infty]$ has a decomposition as a finite sum of quasi-homogeneous derivations, thus by $\deg\partial$ of a quasi-homogeneous derivation we meant to the $e\in S\subset S_\infty$ such that $\deg \partial'=e$ where $\partial(\chi^a)=\partial'(\chi^a)+c_{\partial(a)}\chi^\infty$. This section we make a complete classification of topologically integrable quasi-homogeneous derivation on $\mathbb{C}[S_\infty]$ motivated by the next example:
\begin{example}\label{2-derivation}
With $\mathbb{C}[S_\infty]$ as Example \ref{example derivation} endowed with the topology induced by the linear system $\{\mathfrak{a}_i\}_{i\in \NN}$ where $\mathfrak{a}_i$ are the ideals $\la\{\chi^j-\chi^\infty\mid j>i\}\ra$. We can define the derivation 

\[
\begin{array}{rcl} 
 \partial_2\colon \mathbb{C}[S_\infty]& \rightarrow &\mathbb{C}[S_\infty] \\ 
 \chi^1&\mapsto &\chi^{2}-\chi^\infty\\
 \chi^\infty&\mapsto &0
 \end{array}
\]
 Effectively is a derivation since $\partial'=\chi^{2}\partial$ where $\partial$ is the derivation defined in Example \ref{example derivation}. We can see that is quasi-homogeneous with $\deg\partial'=2$ and topologically integrable since $\partial'^{(n)}(\chi ^a+\mathfrak{a}_i)\subset \mathfrak{a}_i$ for all $n\geq 1$.
 In polynomial terms, the derivation $\partial_2$ has associated the derivation $x^2\dfrac{d}{dx}\colon \CC[x]\rightarrow\CC[x]$. In general, as we will prove more general, $x^l\dfrac{d}{dx}\colon \CC[x]\rightarrow\CC[x]$ give us a topologically integrable derivation on $\CC[S_\infty]$ if and only if $l>1$. 
 \end{example}
 \begin{example}\label{0-derivation}
 In general, it is not true that homogeneous locally nilpotent derivations are topologically integrable. Again with $S=\NN$ and the same linear system of Example \ref{2-derivation} we define the locally nilpotent derivation
 \[
\begin{array}{rcl} 
 \partial\colon \mathbb{C}[S_\infty]& \rightarrow &\mathbb{C}[S_\infty] \\ 
 \chi^1&\mapsto &1-\chi^\infty\\
 \chi^\infty&\mapsto &0
 \end{array}
\]
 
 $\partial$ does not verify (P.2) of Lemma \ref{equivalence}, because $\partial^{(n)}(\chi^a)=\dfrac{a!}{(a-n)!}(\chi^{a-n}-\chi^\infty)$ for all $n\leq j$ thus for all $\mathfrak{a}_i$ and $\mathfrak{a}_j$ (we can assume $i<j$) with $j<a$ and $n=a-i$ we have $\partial^{(a-i)}(\chi^a-\chi^\infty)=\dfrac{a!}{i!}(\chi^{i}-\chi^\infty)$ therefore $\partial^{(n)}(\mathfrak{a}_j)$ is not a subset of $\mathfrak{a}_i$.
\end{example}

In \cite[Theorem 2.7]{L10} there is a complete classification of homogeneous locally nilpotent derivation on $\mathbb{C}[S]$ when $S$ is saturated. When $S$ is saturated, an element $e\in M$ is called a Demazure root if there exists $\rho\in \sigma(1)$ such that $\la e,\rho\ra=-1$ and $\la e,\rho '\ra \geq 0$ for all $\rho'\in \sigma(1)$ with $\rho'\neq \rho$. We are going to denote $\mathcal{R}(S)$ to the set of Demazure roots associated to $S$ and by $\mathcal{R}_\rho(S)$ we will refer to the set of Demazure roots associated to $S$ with distinguished ray $\rho$. For all $e\in \mathcal{R}_\rho(S)$ we can define the $M$-homogeneous locally nilpotent derivation $\partial_e\colon\mathbb{C}[S]\rightarrow \mathbb{C}[S]$ defined by $\partial_e(\chi^a)=\la a,\rho\ra\chi^{a+e}$, furthermore all $M$-homogeneous locally nilpotent derivation on $\mathbb{C}[S]$, different of the $0$ derivation, it is of the form $\lambda\partial_e$ where $\lambda \in \mathbb{C}^*$. When $S$ is not saturated and $S^{sat}$ is the saturation of $S$, a derivation $\lambda\partial_e$ on $\mathbb{C}[S^{sat}]$ with $e\in \mathcal{R}_\rho(S^{sat})$ defines a $M$-homogeneous locally nilpotent derivation on $\mathbb{C}[S]\subset \mathbb{C}[S^{sat}]$ if $\mathbb{C}[S]$ is stabilized by $\partial_e$. Also, as is proved in the next lemma, all $M$-homogeneous locally nilpotent derivation on $\mathbb{C}[S]$ is the restriction of a $M$-homogeneous locally nilpotent derivation on $\mathbb{C}[S^{sat}]$ thus $\deg \partial \in \mathcal{R}(S^{sat})$.

\begin{lemma}
If $\partial\neq 0 \in \operatorname{LND}(\mathbb{C}[S])$ is $M$-homogeneous of $\operatorname{deg}\partial=e$, then there exists a $M$-homogeneous $\partial'\in\operatorname{LND}(\mathbb{C}[S^{sat}])$ such that $\partial'|_{\mathbb{C}[s]}=\partial$. Furthermore $e\in \mathcal{R}(S)$.
\end{lemma}
\begin{proof}
By theorem in \cite{S66} we can to extend $\partial$ to a derivation $\partial'$ on $\mathbb{C}[S^{sat}]$ additionally $\partial'$ is locally nilpotent derivation by \cite[Theorem 2.2]{V69}. Now, for all $a\in S^{sat}$ there exists a minimal $k$ positive integer such that $k a\in S$, then $\partial'(\chi^a)$ verifies $\partial'(\chi^a)=\partial(\chi^a)$ if $a\in S$, $\partial'(\chi^a)=0$ if $k a\in \operatorname{ker}\partial$ and for the other cases $\partial'(\chi^{k a})=k\chi^{(k-1) a}\partial'(\chi^a)$ but also $\partial(\chi^{k a})=c\chi^{k a+e}$ where $c\in \mathbb{C}^*$, if we see the elements $c\chi^{k a+e}$ and $k\chi^{(k-1) a}\partial'(\chi^a)$ in $\mathbb{C}[M]$ we can conclude that $\partial'(\chi^a)=c'\chi^{a+e}$ where $c'\in \mathbb{C}^*$ thus $\partial'$ is a homogeneous locally nilpotent derivation and by \cite[Theorem 2.7]{L10} $e\in \mathcal{R}(S^{sat})$.
\end{proof}

The previous lemma together \cite[Theorem 2.7]{L10} give us a complete description of homogeneous locally nilpotent derivations on $\mathbb{C}[S]$ in term of a subset of $\mathcal{R}(S^{sat})$. 
\begin{definition}
Let S be an affine semigroup, a Demazure root on $S$ is a Demazure root associated to $S^{sat}$ such that $\mathbb{C}[S]$ is stabilized by $\partial_e$. By $\mathcal{R}(S)$ we mean to the Demazure roots of $S$ and by $\mathcal{R}_\rho(S)$ we mean to the Demazure roots with $\rho$ the distinguished ray.
\end{definition}
A natural question, motivated by Example \ref{0-derivation} and the possibility to obtain a description of homogeneous locally nilpotent derivation on $\CC[S]$, is: \emph{When a homogeneous locally nilpotent derivation on $\mathbb{C}[S]$ gives us a quasi-homogeneous topologically integrable derivation on $\CC[S_\infty]$?} To obtain an answer to the question allow us fix some notations. By Example \ref{2-derivation}, we will give an explicit description of the linear system $\{\mathfrak{a}_i\}_{i\in \NN}$ on $\CC[S_\infty]$ defined by $\ker\psi_i$. If $\mathfrak{a}_i=\la \{\chi^a-\chi^\infty\mid \mathfrak{s}(a)>i\}\ra$, or equivalently $\mathfrak{a}_i=\la \{\chi^a-\chi^\infty\mid a\notin H_i\}\ra$ we have that $\mathfrak{a}_i=\ker \psi_i$. Obviously $\mathfrak{a}_i\subset \ker \psi_i$ on the other hand, the contention $\ker\psi_i\subset \mathfrak{a}_i$ is not so immediate but is a combinatorial proof proved in the next lemma: 

\begin{lemma}
With the previous definitions, $\ker\psi_i\subset \mathfrak{a}_i$.
\end{lemma}
\begin{proof}

If $f\in\ker\psi_i\subset \CC[S_\infty]$, with $f=\sum_{j=1}^m \lambda_{b_j}\chi^{b_j}+\lambda_\infty\chi^\infty$ then $b_j\notin S_i$ for all integer $j$ with $1\leq j\leq m$. If $\lambda_\infty\neq 0$ then 
$\sum_{j=1}^m \lambda_{b_j}+\lambda_\infty=0$ thus $f=\sum_{j=1}^m \lambda_{b_j}\chi^{b_j}+\lambda_\infty\chi^\infty=\sum_{j=1}^m \lambda_{b_j}\chi^{b_j}-(\sum_{j=1}^m \lambda_{b_j})\chi^\infty=\sum_{j=1}^m\lambda_{b_j}(\chi^{b_j}-\chi^\infty)$ therefore $f\in \mathfrak{a}_i$ since $\chi^{b_j}-\chi^\infty$ is an element in $\mathfrak{a}_i$ for all $j$. If $\lambda_\infty=0$ then we can rewrite $f=f_1-f_2$ where $f_1=\sum_{j=1}^{m_1} \lambda_{c_j}\chi^{c_j}$ and $f_2=\sum_{j=1}^{m_2} \lambda_{d_j}\chi^{d_j}$ where $\sum_{j=1}^{m_1} \lambda_{c_j}=\sum_{j=1}^{m_2} \lambda_{d_j}$ then $f=\sum_{j=1}^{m_1} \lambda_{c_j}\chi^{c_j}-\sum_{j=1}^{m_2} \lambda_{d_j}\chi^{d_j}+(\sum_{j=1}^{m_2} \lambda_{d_j}-\sum_{j=1}^{m_1} \lambda_{c_j})\chi^\infty=\sum_{j=1}^{m_1} \lambda_{c_j}(\chi^{c_j}-\chi^{\infty})-\sum_{j=1}^{m_2} \lambda_{d_j}(\chi^{d_j}-\chi^\infty)$ thus analogous with the case $\lambda_\infty\neq 0$ the element $f$ is in $\mathfrak{a}_i$.
\end{proof}

Now we can see if $\partial$ is a quasi-homogeneous lnd just we need to verify (P.2) of Lemma \ref{equivalence} because $\partial(\lambda)=0$ for all $\lambda \in \CC$ and (P.1) is verified since there exists $n$ such that $\partial^{(n)}(f)=0$.

\begin{proposition}\label{irreducible}
Let $\partial$ be a quasi-homogeneous locally nilpotent derivation of degree $e\in \mathcal{R}\rho(S)$ on $\mathbb{C}[S_\infty]$.
$\partial$ is topologically integrable if and only if $-e\notin S$.
\end{proposition}

\begin{proof}
 If $-e\in S$ we will prove that for all $i$ and $j$ (we can assume $j\geq i$) there exists $n$ such that $\partial^{(n)}(\mathfrak{a}_j)$ is not a subset of $\mathfrak{a}_i$. First to note if $e$ is a Demazure root then $-e\in \mathscr{H}$ because if $-e\notin \mathscr{H}$ then $-e=a+b$ where $a,b\in S$ so $e=-a-b$ but $\la e, \rho\ra =-1$ thus 
 $-\la a,\rho\ra -\la b,\rho\ra=-1$ but it is possible just if $a=0$ or $b=0$. Now we fix the element {$\chi^{(j+1)(-e)}-\chi^\infty\in \mathfrak{a}_j$} and $n=j-i+1$ the $n$-th iteration $\partial^{(n)}(\chi^{(j+1)(-e)}-\chi^\infty)=\lambda_n(\chi^{(j+1)(-e)+n e}-\chi^\infty)=\lambda_n(\chi^{(j+1)(-e)+(i-(j+1))(-e)}-\chi^\infty)=\lambda_n(\chi^{i(-e)}-\chi^\infty)\notin \mathfrak{a}_i$ since $-e\in \mathscr{H}$. 

In the other direction, if $-e\notin S$ we will prove that $\partial^{(n)}(\mathfrak{a}_i)\subset \mathfrak{a}_i$ for all $i$ and for all $n$ positive integers. 
 First we can see that $\mathfrak{s}(a+e)\geq \mathfrak{s}(a)$ for all $a\in \mathscr{H}$ such that $a\notin \rho^* \cap S$. Also if $a\in S$ with $a\notin \rho^*\cap S$, there exists $a'\in \mathscr{H}$ such that $a'\notin\rho^*\cap S$ and $a=a'+b$ for some $b\in S$ verifying $\mathfrak{s}(b)\geq \mathfrak{s}(a)-1$ thus $\mathfrak{s}(a+e)=\mathfrak{s}(b+a'+e)\geq \mathfrak{s}(b)+\mathfrak{s}(a'+e)\geq \mathfrak{s}(a)-1+ \mathfrak{s}(a'+e) \geq \mathfrak{s}(a)$. Now if $f=\chi^{a}-\chi^\infty$ is an element in $\mathfrak{a}_i$, there exists a minimal $l$ such that $\partial^{(n)}(f)=0\in \mathfrak{a}_i$ for all $n>l$ when $n\leq l$ we have that $\partial^{(n)}(\chi^a-\chi^\infty)=\lambda_n(\chi^{a+n e}-\chi^\infty)$ where $\lambda_n\neq 0\in \CC$ depending of $n$ and $a$. Since $a+(n-1) e\notin \rho^*\cap S$ for all $n\leq l$ previously we proved that $\mathfrak{s}(a+n e)\geq \mathfrak{s}(a+(n-1) e)$ thus we conclude that $\partial^{(n)}(\chi^a-\chi^\infty)\in \mathfrak{a}_i$ for all $n$. Finally we can see that if $f=\chi^a-\chi^\infty\in \mathfrak{a}_i$ and $g\in \mathbb{C}[S]$ the $n$-th iteration is $\partial^{(n)}(gf)=\sum_{i=0}^n\binom{n}{i}\partial^{(i)}(f)\partial^{(n-1)}(g)\in \mathfrak{a}_i$ and together of the linearity of $\partial$, $\partial^{(n)}(\mathfrak{a}_i)\subset \mathfrak{a}_i$ for all $i$ and for all $n$ positive integers. 
\end{proof}

\begin{definition}
A Demazure root of $S$ is said to be reducible, if there exist $a\neq 0\in S$ and $e'\neq e\in \mathcal{R}(S)$ such that $e=e'+a$. If $e$ is not reducible will be called an irreducible root.
\end{definition}

\begin{lemma}
Let $\partial$ be a quasi-homogeneous locally nilpotent derivation of degree $e\in \mathcal{R}\rho(S)$. If $e$ is reducible, then $\partial$ is topologically integrable.
\end{lemma}

\begin{proof}
if $e$ is reducible, $e=e'+a$ for some $a\in S$ and $e'\neq e\in \mathcal{R}(S)$. If $-e=-(e'+a)\in S$ we can see that $e'+a\notin \rho^*\cap S$ since $\la e'+a,\rho \ra=1$, then $e'+[-(e'+a)]=-a\in S$ a contradiction because $S$ is pointed. 
\end{proof}

As we saw in Example \ref{2-derivation} also there exists quasi-homogeneous topologically integrable derivation but they are not lnd. The follow theorem finish the classification of quasi-homogeneous topologically integrable derivation. 
 
\begin{theorem}\label{clasification}
Let $\partial$ be a not lnd quasi-homogeneous derivation on $\CC[S_\infty]$ of degree $e\in M$. $\partial$ is topologically integrable if and only if $e\in S\bs\{0\}$.
\end{theorem}
 
\begin{proof}
In general $\partial(\lambda)=0$ for all $\lambda$ in $\CC$ since $\partial(\lambda)=\partial'(\lambda)$ where $\partial'$ is the respective derivation on $\CC[S]$ associated to $\partial$.
Now, we need to verify (P.1) and (P.2) of Lemma \ref{equivalence} are verified if and only if $e\in S\bs\{0\}$. First we assume $e\in S\bs\{0\}$,
to prove (P.1), since $\CC$-linearity of $\partial$, just we need to verify with $\chi^a$ where $a\in S_\infty$. When $a=\infty$, $\partial(\chi^a)=0$ so $\partial^{(n)}(\chi^a)\in \mathfrak{a}_i$ for all $n$ positive integer and $i$ no negative integer. It is the same with $a=0$. If $a\neq \infty$ and $a\neq 0$, the element $\partial^{(n)}(\chi^a)=(1-\chi^\infty)\lambda_n\chi^{a+n e}=(\chi^{n e}-\chi^\infty)\lambda_n\chi^a$ for some $\lambda_n\in \CC$ depending of $n$ a positive integer. We can see that $n e\notin H_i$ for $i<n$ thus for all $i$ with $n_0>i$ the elements $\partial^{(n)}(\chi^a)\in \mathfrak{a}_i$ for all $n\geq n_0$. To verify (P.2), with $a \notin H_i$ the element $\partial^{(n)}(\chi^a-\chi^\infty )=\partial^{(n)}(\chi^a)=(1-\chi^\infty)\lambda_n\chi^{a+n e}=(\chi^{a+n e}-\chi^\infty)\lambda_n$ since $S$ is pointed, $a+n e\notin H_i$ for all $n$. Now if $f\in \mathfrak{a}_i$ is a general element, that is, $f=\sum_{k=1}^rf_k(\chi^{b_k}-\chi^\infty)$ with $b_k\notin H_i$. Again by the linearity of $\partial$ just it is necessary to assume $f=g(\chi^a-\chi^\infty)$. $\partial^{(n)}(g(\chi^a-\chi^\infty))=\sum_{k=0}^n\binom{n}{k}\partial^{(n-k)}(g)\partial^{(k)}(\chi^a-\chi^\infty)$ thus $\partial^{(n)}(\mathfrak{a}_i)\subset\mathfrak{a}_i$ for all $n$ since $\partial^{(k)}(\chi^a-\chi^\infty)$ are elements in $\mathfrak{a}_i$. If $e=0$, by the relation between derivation proved in Proposition \ref{LND} $\partial$ is not locally nilpotent since the derivation associated on $\CC[S]$ it is not locally nilpotent (see \cite[Theorem 2.7]{L10}) so there exists $a\in S_\infty$ such that $\partial^{(n)}(\chi^a)\neq 0$ also $a\in H_i$ for some $i$ 
thus $\partial^{(n)}(\chi^a)=(\chi^a-\chi^\infty)\lambda_n\notin \mathfrak{a}_i$ for all $n$ therefore $\partial$ it does not verify (P.1) and it is not topologically integrable.

\end{proof}

\end{document}